\documentclass[12pt]{amsart}
\usepackage{amsmath,amssymb,amsthm,amscd,verbatim}
\usepackage{graphicx}
\usepackage{epsfig}
\usepackage{hyperref}

\newtheorem{thm}{Theorem}
\newtheorem{lem}[thm]{Lemma}
\newtheorem{conj}[thm]{Conjecture}

\newtheorem{cor}[thm]{Corollary}
\newtheorem*{thm*}{Theorem}

\theoremstyle{definition}

\theoremstyle{remark}

\newcommand{\cR}{\mathcal{R}}
\newcommand{\cI}{\mathcal{I}}

\newcommand{\cS}{\mathcal{S}}
\newcommand{\cF}{\mathcal{F}}

\newcommand{\NOS}{\mathbb{N}}
\newcommand{\OS}{\mathbb{S}}

\newcommand{\bo}{\ensuremath{\mathrm{box}}}

\newcommand{\dist}{\ensuremath{\mathrm{dist}}}

\renewcommand{\le}{\leqslant}

\renewcommand{\ge}{\geqslant}

\begin{document}
\title{Box representations of embedded graphs}
\author{Louis Esperet} \address{Laboratoire G-SCOP (CNRS,
  Universit\'e Grenoble-Alpes), Grenoble, France}
\email{louis.esperet@grenoble-inp.fr}

\thanks{The author is partially supported by ANR Project STINT
  (\textsc{anr-13-bs02-0007}), and LabEx PERSYVAL-Lab
  (\textsc{anr-11-labx-0025}).}

\date{}
\sloppy
\begin{abstract}
A \emph{$d$-box} is the cartesian product of $d$ intervals of $\mathbb{R}$
and a \emph{$d$-box representation} of a graph $G$ is a representation
of $G$ as the intersection graph of a set of $d$-boxes in
$\mathbb{R}^d$. It was proved by Thomassen in 1986 that every planar
graph has a 3-box representation. In this paper we prove that every graph
embedded in a fixed orientable surface, without short non-contractible cycles,
has a 5-box representation. This directly implies that there is a function $f$,
such that in every graph of genus
$g$, a set of at most $f(g)$ vertices can be removed so that the resulting graph has a
5-box representation. We show that such a function $f$ can be made
linear in $g$. Finally, we prove that for any proper minor-closed class
$\cF$, there is a constant $c(\cF)$ such that every graph of $\cF$
without cycles of length less than $c(\cF)$ has a 3-box
representation, which is best possible.
\end{abstract}

\maketitle

\section{Introduction}

For $d\ge 1$, a \emph{$d$-box} is the cartesian product $I_1\times I_2
\times \cdots \times I_d$ of $d$ intervals of $\mathbb{R}$. A
\emph{$d$-box representation} of a graph $G=(V,E)$ is a collection
$\cR=(B_v)_{v \in
  V}$ of $d$-boxes in $\mathbb{R}^d$ such that any two boxes $B_u$ and
  $B_v$ intersect if and only if the corresponding vertices $u$ and
  $v$ are adjacent in $G$. In other words, $G$ is the intersection
  graph of the boxes $(B_v)_{v \in
  V}$. The \emph{boxicity} of a graph $G$, denoted by $\bo(G)$ and introduced by Roberts in 1969~\cite{Rob69}, is the smallest integer $d$
such that $G$ has a $d$-box representation.

It was proved by Thomassen in 1986 that planar graphs have boxicity
at most 3~\cite{Tho86}, which is best possible (as shown by the planar
graph
obtained from a complete graph on 6 vertices by removing a perfect
matching~\cite{Rob69}). It is natural to
investigate how this result on planar graphs extends to graphs
embeddable on surfaces of higher genus. Let $\bo(g)$ be the supremum
of the boxicity of all graphs embeddable in a surface of Euler genus
$g$. The result of Thomassen on planar graphs was extended
in~\cite{EJ13} by showing that for any $g\ge 0$, $\bo(g)\le 5g+3$ (prior to this result, it was not known whether
$\bo(g)$ was finite). In~\cite{Esp16}, we
proved the existence of two constants $c_1,c_2>0$, such that $c_1 \sqrt{g  \log g} \le
\bo(g) \le c_2\sqrt{g} \log g$. The proof of the upper bound relies on
a connection between boxicity and acyclic coloring established
in~\cite{EJ13}. It was noted there that using this connection and a
result of Kawarabayashi and Mohar~\cite{KM10}, it could be proved that there
is a function $f$ such that any graph embedded in a surface of Euler
genus $g$, such that all non-contractible cycles have length at least
$f(g)$, has boxicity at most 42. The main result of this paper is to
reduce this bound to 5 for orientable surfaces (Theorem~\ref{thm:locpla}). Similar ideas are
then used to show that toroidal graphs have boxicity at most 6, and
toroidal graphs without non-contractible triangles have boxicity at
most 5 (Theorem~\ref{thm:tor}). This improves on~\cite{EJ13}, where it was proved that
toroidal graphs have boxicity at most 7, while there are toroidal
graphs of boxicity 4.

An immediate consequence of Theorem~\ref{thm:locpla} is that if $G$
has genus $g$, then a set of at most $\sum_{i=1}^g f(i)$ vertices can be removed in
$G$ so that the resulting graph has boxicity at most 5. However, the
bound we obtain for $f(g)$ in Theorem~\ref{thm:locpla} is exponential in $g$. We show the
following improvement: if $G$ is embedded in a surface of Euler
genus $g>0$, then a set of at most $60g-30$ vertices can be removed in
$G$ so that the resulting graph has boxicity at most 5
(Theorem~\ref{thm:linex1}). Note that this result is proved for any
surface, orientable or not.

In~\cite{EJ13}, it was proved that there is a function $c$ such
that if $G$ is embedded in a surface of Euler
genus $g$, and has no cycle of length less than $c(g)$, then $G$ has
boxicity at most 4. Here, we show there is a function $c$ such
that if $G$ has no $K_t$-minor and no cycle of length less than $c(t)$,
then $G$ has boxicity at most 3 (Corollary~\ref{cor:minor}). This is
best possible already for $t=6$. This result follows from a more
general theorem on path-degenerate graphs
(Theorem~\ref{thm:6path}). This general result also implies, together
with earlier results from~\cite{GGH01}, that there is a constant $c$
such that any graph embeddable in a surface of Euler genus $g$, with
no cycle of length less than $c \log g$, has boxicity at
most 3 (Corollary~\ref{cor:girth}). This is best possible up to the choice of the constant $c$.

\medskip

Some of the results we will use originate from the proof of
Thomassen~\cite{Tho86} that planar graphs have boxicity at most 3, without being explicitly stated there. In
Section~\ref{sec:pla}, we explain how these results can be derived
from~\cite{Tho86}, which might be of independent interest.
We will then prove the main results of this paper in Sections~\ref{sec:locpla},
\ref{sec:linext}, and~\ref{sec:girth}. In the remainder of this section, we review the necessary
background on boxicity and graphs on surfaces. We then give a simple
proof that embedded graphs of large edge-width have boxicity at most
7 (this will be improved to 5 in Section~\ref{sec:locpla}).

\medskip

\subsection*{Boxicity}

Let $G=(V,E)$ be a graph, and let $\cR=(B_v)_{v \in V}$ be a $d$-box
representation of $G$. For a $d$-box $B_v=I_1\times I_2 \times \cdots
\times I_d$, and an integer $1\le i \le d$, we refer to $I_i$ as \emph{the $i$-th interval of $v$}. For $1\le i \le
d$, let $\cI_i$ be the interval representation consisting of all $i$-th 
intervals of the vertices of $G$ in $\cR$. Each interval
representation $\cI_i$, $1\le i \le d$, corresponds to an interval
graph $G_i$ with vertex-set $V$. Observe that each graph $G_i$ is a
supergraph of $G$, and any two $d$-boxes
$B_u$ and $B_v$ intersect if and only if for all $1 \le i \le d$ the
intervals corresponding to $u$ and $v$ intersect in $\cI_i$, or
equivalently, if the
vertices $u$ and $v$ are adjacent in the interval graph $G_i$. 

For two graphs $G_1=(V,E_1)$ and $G_2=(V,E_2)$ on the same vertex-set
$V$, the \emph{intersection $G_1\cap G_2$} of $G_1$ and $G_2$ is
defined as the graph $G=(V,E_1 \cap E_2)$. The discussion above implies
that the boxicity of a graph $G$ can be equivalently defined as the
least $k$ such that $G$ can be expressed as the intersection of $k$
interval graphs on the same vertex-set. In this paper, the two definitions will be used and
we will often switch from one to the other, depending of the
situation, i.e whenever we consider some $d$-box representation
$\cR=(B_v)_{v \in V}$ we will implicitly consider it as a
representation $\cR=(\cI_1,\cI_2,\ldots,\cI_d)$ as defined in the
previous paragraph, and vice-versa. If $\cR$ is a representation of $G$, we will
often say that $\cR$ \emph{represents} $G$, or \emph{induces} $G$.

\subsection*{Graphs on surfaces}

We refer the reader to the book by Mohar
and Thomassen~\cite{MoTh} for more details or any notion not defined
here. All the graphs in this paper are simple (i.e., without loops and
multiple edges). A {\em surface} is a non-null compact connected
2-manifold without boundary. A surface can be orientable or non-orientable. The \emph{orientable
  surface~$\OS_h$ of genus~$h$} is obtained by adding $h\ge0$
\emph{handles} to the sphere; while the \emph{non-orientable
  surface~$\NOS_k$ of genus~$k$} is formed by adding $k\ge1$
\emph{cross-caps} to the sphere. The {\em Euler genus} of a surface
$\Sigma$ is defined as twice its genus if $\Sigma$ is orientable, and as
its non-orientable genus otherwise.

We say that an embedding is \emph{cellular} if every face is
homeomorphic to an open disk of~$\mathbb{R}^2$. Using the fact that the
boxicity of a graph $G$ is at most $k$ if and only if all the
connected components of $G$ have boxicity at most $k$, we will always
be able to assume in this paper (using~\cite[Propositions 3.4.1 and
3.4.2]{MoTh}) that the considered embeddings are cellular (and we will do
so implicitly).

Let $G$ be a graph embedded in a surface of Euler genus $g>0$. The
\emph{edge-width} of $G$ is defined as the length of a smallest
non-contractible cycle of $G$, and the \emph{face-width} of $G$ is defined
as the minimum number of points of $G$ intersected by a
non-contractible curve on the surface.
Given a 2-sided cycle $C$ in $G$, and an orientation of $C$, the set
of neighbors of $C$ incident to an edge leaving the left side of $C$
is denoted by $L(C)$, and the set
of neighbors of $C$ incident to an edge leaving the right side of $C$
is denoted by $R(C)$ (see Figure~\ref{fig:lrc}, left, for an
example in a triangulation). Note that $L(C)$ and $R(C)$ might not be
disjoint. Given two non-contractible cycles $C_1$ and $C_2$,
$\dist(C_1,C_2)$ denotes the minimum distance between a vertex of
$C_1$ and a vertex of $C_2$ in $G$. If $C$ is a non-contractible
2-sided cycle, we also define $\dist(C,C)$ as the minimum length
(number of edges) of a path starting with an edge incident to $C$ and
$L(C)$, and ending with an edge incident to $C$ and $R(C)$ (note that
$\dist(C,C)$ does not depend of the chosen orientation of $C$).

\begin{figure}[htbp]
\begin{center}
\includegraphics[scale=1.4]{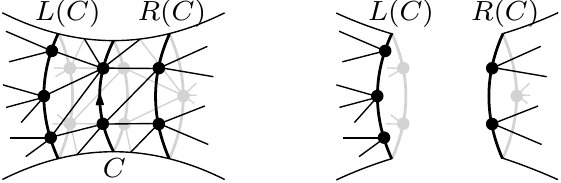}
\caption{The sets $L(C)$ and $R(C)$ in a triangulation (left) and the
  embedded graph resulting from the removal of $C$ (right).\label{fig:lrc}}
\end{center}
\end{figure}

Let $G$ be a triangulation of $\OS_g$. 
A collection of pairwise disjoint non-contractible cycles $C_1,\ldots,C_g$ of $G$ is \emph{planarizing}
if the graph obtained from $G$ by removing the vertices of
$C_1,\ldots,C_g$ is planar. If for any $i\le j$, $\dist(C_i,C_j)\ge d$, then the planarizing collection $C_1,\ldots,C_g$ is said to have
\emph{minimum distance} at least $d$.

The following was proved by Thomassen~\cite{Tho93}:

\begin{thm}{\cite{Tho93}}\label{thm:placc}
Let $d,g$ be integers, and let $G$ be a triangulation of $\OS_g$ of
edge-width at least $8(d+1)(2^g-1)$. Then $G$ contains a planarizing
collection of induced cycles with mininum distance at least $d$.
\end{thm}

\subsection*{Triangulations of large edge-width have boxicity at most 7}

To give a taste of the proofs of the main results is this paper, we
now prove that \emph{if $G=(V,E)$ is a triangulation of $\OS_g$, $g\ge
1$, with edge-width at  least $40(2^g-1)$, then $G$ has boxicity at
most 7.}

\smallskip

Let $C_1,\ldots,C_g$ be a planarizing
collection of induced cycles with mininum distance at least 4 (whose
existence follows from Theorem~\ref{thm:placc}) in $G$. We denote by
$C$ the set of vertices lying in one of the $C_i$'s. We also denote by
$N$ the set of vertices not in $C$ with at least one neighbor in $C$,
and set $R=V\setminus (C\cup N)$. 

Since the collection
$C_1,\ldots,C_g$ is planarizing, the set $V\setminus C=N\cup R$
induces a planar graph and therefore has a 3-box representation
$(\cI_1,\cI_2, \cI_3)$ by the
result of Thomassen~\cite{Tho86}. We extend this representation to $C$ by
mapping all the vertices of $C$ to a big 3-box containing all the
other 3-boxes of the representation. Since the distance in $G$ between any two
cycles $C_i,C_j$ is at least 4, there is no edge between a neighbor
of $C_i$ and a neighbor of $C_j$. As a consequence, the set $C\cup N$
induces a disjoint union of $g$ planar graphs (and is therefore
planar). Let $(\cI_4,\cI_5, \cI_6)$ be a 3-box representation of this
graph. We extend this representation to $R$ by
mapping all the vertices of $R$ to a big 3-box containing all the
other 3-boxes of the representation $(\cI_4,\cI_5, \cI_6)$. Finally,
let $\cI_7$ be the interval graph mapping all the vertices of $C$ to
$\{0\}$, all the vertices of $N$ to $[0,1]$, and all the vertices of
$R$ to $\{1\}$. We now prove that $\cR=(\cI_1,\cI_2,\ldots,\cI_7)$
induces $G$. Each interval graph $\cI_i$ is clearly a supergraph of
$G$, so we only need to show that every pair of non-adjacent vertices
$u,v$ in $G$ is non-adjacent in at least one of the $\cI_i$'s. This
follows from the definition of $(\cI_1,\cI_2, \cI_3)$ if $u,v\in N\cup
R$, and from the definition of $(\cI_4,\cI_5, \cI_6)$ if $u,v\in C\cup
N$. Finally if $u\in C$ and $v\in R$ then they are non-adjacent in
$\cI_7$, as desired.\hfill $\Box$

\medskip

In Section~\ref{sec:locpla} we will show that this argument
can be applied to any embedded graph (not only triangulations) and
that the bound on the boxicity can be
improved to 5. The idea will be to group $\cI_3$, $\cI_6$, and $\cI_7$
in a single interval graph. But for this we first need to see
how much flexibility we have in choosing 3-box representations for
planar graphs, which
is the topic of the next section.

\section{Planar graphs}\label{sec:pla}

Recall that Thomassen~\cite{Tho86} proved that planar graphs have
boxicity at most 3. He actually proved significantly stronger results,
which we will need here. 

A $d$-box is \emph{non-degenerate} if it is the cartesian product of $d$
intervals of positive length. A representation of a graph $G$ as the
intersection of $d$-boxes is \emph{strict} if (1) the boxes are
non-degenerate, (2) the interiors of the
boxes are pairwise disjoint, and (3) the intersection of any two boxes is
a non-degenerate $(d-1)$-box.

\begin{thm}{\cite{Tho86}}\label{thm:plasep} A graph $G$ has a strict
  2-box representation if and only if $G$ is a proper subgraph of a
  4-connected planar triangulation.
\end{thm}

Let $G$ be a graph embedded in some surface. A triangle of $G$ is said
to be \emph{facial} if it bounds a face of $G$.
We will use the following immediate corollary of Theorem~\ref{thm:plasep}:

\begin{cor}\label{cor:plasep} 
  Let $G$ be a plane graph such that all the triangles of $G$ are
  facial. If $G$ is not a triangulation, then $G$ has a strict 2-box representation.
\end{cor}

\begin{proof}
  Let $H$ be the graph obtained from $G$ by doing the following, for
  every non-triangular face $f$ of $G$: add a cycle $C_f$ of length
  $d(f)$ (the number of edges in a boundary walk of $f$) inside $f$,
  join one vertex of $C_f$ to each of the other vertices of $C_f$
  (with the newly added edges staying inside the region bounded by
  $C_f$), and then join each vertex of $C_f$ to two consecutive
  vertices in a boundary walk of $f$. This can be done in such a way
  that $H$ is a (simple) triangulation. Since all the triangles of $G$
  are facial, $H$ is a triangulation without separating
  triangle. Observe that $H$ contains at least 5 vertices, so $H$ is a
  4-connected triangulation. The result now follows directly from
  Theorem~\ref{thm:plasep}.
\end{proof}

Consider a strict 3-box representation $(B_v)_{v\in V}$ of a planar
graph $G=(V,E)$, and let $uvw$ be a triangle of $G$. Then $uvw$ has
an \emph{empty inner corner} in the representation if there is a small 3-box
$C$, whose interior is disjoint from the interior of all the boxes
$(B_v)_{v\in V}$, such that the following holds: there is a corner $c$
of $C$, and three faces $f_u,f_v,f_w$ of $C$ containing $c$, such that
$C\cap B_u=f_u$, $C\cap B_v=f_v$, and $C\cap B_w=f_w$ (see
Figure~\ref{fig:corner}). In particular
$c=f_u\cap f_v \cap f_w \in B_u\cap B_v \cap B_w$.  Note that if a
triangle has an empty inner corner, then the corresponding 3-box $C$
defined above can be made arbitrarily small.

\begin{figure}[htbp]
\begin{center}
\includegraphics[scale=1.2]{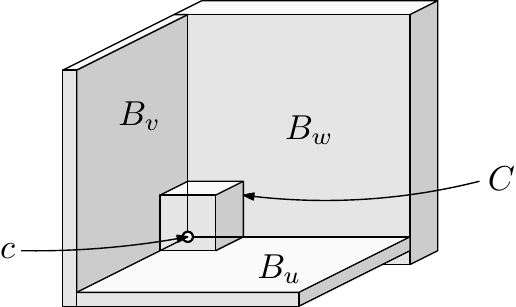}
\caption{An empty inner corner in a strict 3-box representation. \label{fig:corner}}
\end{center}
\end{figure}

\medskip

We say that three intervals $I_1,I_2,I_3$ are \emph{strictly
  overlapping} if $I_1 \cap I_2 \cap I_3$ is an interval of positive
length, and none of the three intervals is contained in another one of
the three.

 The following
simple lemma about empty inner corners will be particularly useful.

\begin{lem}\label{lem:corner}
  Let $(\cI_1, \cI_2, \cI_3)$ be a strict 3-box representation of a
  planar graph $G$ such that $(\cI_1, \cI_2)$ is a strict 2-box
  representation of $G$. Then any
  triangle $uvw$ of $G$ such that the intervals of $u,v,w$ are strictly
  overlapping in $\cI_3$ has an empty inner corner in
  $(\cI_1, \cI_2, \cI_3)$.
\end{lem}

\begin{figure}[htbp]
\begin{center}
\includegraphics[scale=1]{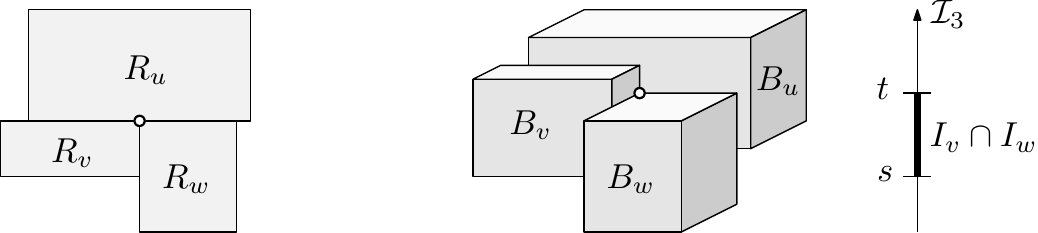}
\caption{The 2-boxes of a triangle
  $uvw$ in
  a strict 2-box representation (left), and an empty inner corner
  (depicted by a white dot) in the corresponding strict 3-box
  representation if the intervals of the third dimension are strictly
  overlapping (right). For the sake of readability, we write $B_x$
  instead of $R_x \times I_x$. \label{fig:over}}
\end{center}
\end{figure}

\begin{proof}
  Let $uvw$ be a triangle of $G$. Without loss of generality, in
  $(\cI_1,\cI_2)$, $u,v,w$ are mapped to rectangles $R_u,R_v,R_w$ such
  that the point $(c_1,c_2)=R_u\cap R_v \cap R_w$ is a corner of $R_v$ and
  $R_w$, but lies on a side of $R_u$  (see
Figure~\ref{fig:over}, left). Moreover, no rectangle of
$(\cI_1,\cI_2)$ other than
$R_u,R_v,R_w$ contains $(c_1,c_2)$. Let $I_u,I_v,I_w$
  be the intervals corresponding to $u,v,w$ in $\cI_3$. Since
  $I_u,I_v,I_w$ are strictly overlapping, $I_u \cap I_v \cap I_w$ is an interval of
  positive length, and none of the three intervals is contained in
  another one. Let $[s,t]=I_v \cap I_w$. Observe that the interior of $I_u$
  intersects at least one of $s,t$ (say $t$ by symmetry), otherwise
  $I_u$ would not intersect $I_v$ or $I_w$ (in an interval of positive
  length), or would be contained in $I_v$ or $I_w$. It is then easy to
  check that the triangle $uvw$ has an empty inner corner in
  $(\cI_1,\cI_2,\cI_3)$, touching the point with coordinate
  $(c_1,c_2,t)$ (see
Figure~\ref{fig:over}, right).
\end{proof}

The following result was proved by Thomassen (see Theorem 4.5 in
\cite{Tho86}), using a strong variant of Theorem~\ref{thm:plasep}.

\begin{thm}{\cite{Tho86}}\label{thm:tho45}
  Every planar triangulation $G$ has a strict 3-box representation
  such that every facial triangle of $G$, except one prescribed
  triangle, has an empty inner corner.
\end{thm}

The proof of Thomassen~\cite{Tho86} indeed shows a slightly stronger
result, which will be extremely useful in this paper.

\begin{thm}{\cite{Tho86}}\label{thm:tho45v2}
  Let $G$ be a planar triangulation $G$, with outerface $uvw$, and let
  $B_u,B_v,B_w$ be any strict 3-box representation of $u,v,w$ such
  that $uvw$ has an empty inner corner $C$. Then $B_u,B_v,B_w$ extends
  to a strict 3-box representation $\cR$ of $G$
  such that every facial triangle of $G$, except possibly $uvw$, has an empty
  inner corner in $\cR$. Moreover, all 3-boxes except that of $u,v,w$ are
  included in $C$.
\end{thm}

Let $G$ be a graph embedded in the plane, or in $\OS_g$ with $g>0$,
and such that all the triangles are contractible. Then each triangle
bounds a region homeomorphic to an open disk. This region is called
the \emph{interior} of the triangle. Let $H$ be the graph obtained
from $G$ by removing (the vertices in) the interior of each
triangle (distinct from the outerface of $G$, in case $G$
is embedded in the plane). The graph $H$ is called the \emph{frame} of
$G$. Note that $H$ is embedded in the plane or $\OS_g$ with $g>0$,
with the property that all its triangles are facial. Moreover, if $G$
is embedded in the plane, the outerfaces of $G$ and $H$ coincide.

A triangle in a graph $G$ embedded in the plane is said to be
\emph{internal} if it is distinct from the outerface of $G$. The following direct consequence of Theorem~\ref{thm:tho45v2} will be
repeatedly used in the next section (in conjunction with Corollary~\ref{cor:plasep} and
Lemma~\ref{lem:corner}). 

\begin{cor}\label{cor:plaframe}
  Let $G$ be a planar graph and let $H$ be its frame. Then any
  strict 3-box representation of $H$ in which each internal triangle has
  a fixed empty inner corner can be extended to a strict 3-box
  representation of $G$. Moreover, if $uvw$ is an internal triangle of $H$
  and $x$ is a vertex of $G$ inside the region bounded by $uvw$ in
  $G$, then $x$ is mapped to a 3-box lying in the empty inner corner
  of $uvw$.
\end{cor}

\section{Locally planar graphs and toroidal graphs}\label{sec:locpla}

We are now ready to prove the main result of this paper.

\begin{thm}\label{thm:locpla}
Let $G$ be a graph embedded in $\OS_g$, $g\ge 1$, with edge-width at  least $40(2^g-1)$. Then $G$ has boxicity at most 5.
\end{thm}

\begin{proof}
  It is well-known that if $G$ has edge-width at least $k$, then $G$
  is an induced subgraph of a triangulation of $\OS_g$ of edge-width
  at least $k$ (see for instance Lemma 3.4 in~\cite{Tho93}). Since
  having boxicity at most $k$ is a property that is closed under
  taking induced subgraphs, in the
  following we can assume that $G$ is a triangulation of $\OS_g$.
  Since $G$ has edge-width at least $40(2^g-1)>3$, all the
  triangles of $G$ are contractible.  Let $H$ be the frame of
  $G$. Observe that $H$ is a triangulation of $\OS_g$ of edge-width at
  least $40(2^g-1)$. By
  Theorem~\ref{thm:placc}, $H$ has a collection of planarizing cycles
  $C_1,\ldots,C_g$ with minimum distance at least 4. 

Consider the following 5 sets of vertices of $G$, forming a partition
of $V(G)$:
\begin{enumerate}
\item[$C$]: the union of all cycles $C_1,\ldots,C_g$,
\item[$N$]: the union of all neighbors in $H$ of a vertex of $C$,
\item[$R$]: the vertices of $H$ not in $C\cup N$,
\item[$T_1$]: the vertices of $G$ lying inside a non-facial triangle
that does not intersect $C$,
\item[$T_2$]: the vertices of $G$ lying inside a non-facial triangle
  intersecting $C$.
\end{enumerate}

Observe that since all the triangles of $G$ are contractible, all the
triangles of $H$ are facial (and moreover, the triangles of $H$ are
precisely the faces of $H$, since $H$ is a triangulation).
  Let $H_1$ be the planar graph obtained from $H$ by deleting the
  vertices of $C$ (i.e. $H_1$ is the subgraph of $G$ induced by $N \cup R$). All the triangles of $H_1$ are facial,
  and we claim that $H_1$ is not a triangulation. To see this, observe
  that the deletion of each $C_i$ produces two faces with vertex-set
  $L(C_i)$ and $R(C_i)$, respectively (see Figure~\ref{fig:lrc}, right). The two cycles bounding these
  faces in $H_1$ correspond to two cycles of $H$ that are (freely)
  homotopic to $C_i$, and therefore non-contractible in $\OS_g$. Since
  $H$ has edge-width at least 4, the boundary walks of the two faces of $H_1$ resulting from
  the deletion of $C_i$ have length at least 4, and $H_1$ is not a
  triangulation. By Corollary~\ref{cor:plasep}, $H_1$ has a strict 2-box
  representation $(\cI_1, \cI_2)$. We extend this 2-dimensional
  representation of $H_1$ to $C \cup T_2$ by mapping all the vertices of $C \cup
  T_2$ to a large 2-box containing all the 2-boxes of $(\cI_1,
  \cI_2)$. The images of the vertices of $T_1$ in $(\cI_1,
  \cI_2)$ will be defined later in the proof.

  Let $H_2$ be the subgraph of $H$ induced by $C\cup N$. Since the
  collection $C_1,\ldots,C_g$ has minimum distance at least 4, $H_2$
  is the disjoint union of $g$ graphs, each being a subgraph of
  $H$ induced by $C_i$ together with its neighborhood $N_H(C_i)$, for
  $1 \le i \le g$. As each of these graphs is embedded in a cylinder,
  they are all planar and so is $H_2$. As before, we can argue that all the triangles of $H_2$ are
  facial, and $H_2$ is not a triangulation. It follows from
  Corollary~\ref{cor:plasep} that $H_2$ has a strict 2-box
  representation $(\cI_3, \cI_4)$. We extend this 2-dimensional
  representation of $H_2$ to $R \cup T_1$ by mapping all the vertices of
  $R \cup T_1$ to a large 2-box containing all the 2-boxes of
  $(\cI_3, \cI_4)$. Again, the images of the vertices of $T_2$ in $(\cI_3,
  \cI_4)$ will be defined later.

  For each vertex $v$ of $H$, we choose a real
  $0<\epsilon_v<\tfrac14$, in such way that all the chosen $\epsilon$
  are distinct. Let $\cI_5$ be the following interval graph: each
  vertex $v$ of $C$ is mapped to $[\epsilon_v,2+\epsilon_v]$, each
  vertex $v$ of $N$ is mapped to $[1+\epsilon_v,4+\epsilon_v]$, and
  each vertex $v$ of $R$ is mapped to
  $[3+\epsilon_v,5+\epsilon_v]$. Observe that in the corresponding
  interval graph, $C$, $N$ and $R$ are cliques, $N$ is complete to $C$
  and $R$, while there are no edges between $C$ and $R$.

  In order to complete the description of our 5-box representation
  $(\cI_1,\ldots,\cI_5)$ of $G$, we only need to prescribe
  the images of the vertices of $T_1$ in $(\cI_1, \cI_2,\cI_5)$ and
  the images of the vertices of $T_2$ in $(\cI_3, \cI_4,\cI_5)$. All
  these images will be defined using Corollary~\ref{cor:plaframe}, as we
  now explain.

  Consider first the 3-box representation $(\cI_1,\cI_2,\cI_5)$,
  restricted to the vertices of $H_1$ (i.e. the vertices of
  $N \cup R$). Since $(\cI_1,\cI_2)$ is a strict 2-box representation
  of $H_1$ and any three intervals of $\cI_5$ corresponding to some vertices
  of $N \cup R$ are strictly overlapping, by Lemma~\ref{lem:corner}
  every (facial) triangle of $H_1$ has an
  empty inner corner in $(\cI_1,\cI_2,\cI_5)$. By
  Corollary~\ref{cor:plaframe}, this representation of $H_1$ can be
  extended to $T_1$, so that each vertex $u$ of $T_1$, lying inside
  some facial triangle $xyz$ of $H$, is mapped to a 3-box $B_u$ such
  that all the points of $B_u$ are at $L_\infty$-distance at most
  $\tfrac14$ from the empty inner
  corner of $xyz$. This shows how to extend
  $(\cI_1,\cI_2,\cI_5)$ to $T_1$ (in the sense that the restriction of
  $(\cI_1,\cI_2,\cI_5)$ to $N \cup R\cup T_1$ is a representation of the subgraph of
  $G$ induced by $N \cup R\cup T_1$).

  Similarly, consider the 3-box representation $(\cI_3,\cI_4,\cI_5)$,
  restricted to $H_2$ (i.e. to the vertices of $C \cup N$). As before,
  we can prove using Lemma~\ref{lem:corner} that all facial triangles of $H_2$ have an empty inner
  corner in this representation, and it follows from
  Corollary~\ref{cor:plaframe} that this representation of $H_2$ can be
  extended to $T_2$, so that each vertex $u$ of $T_2$, lying inside
  some facial triangle $xyz$ of $H$, is mapped to a 3-box $B_u$ such
  that all the points of $B_u$ are at $L_\infty$-distance at most
  $\tfrac14$ from the empty inner corner of $xyz$. Recall that by
  definition of $T_2$, such a triangle $xyz$ intersects $C$, and
  therefore $xyz$ is also a facial triangle of $H_2$. This shows how to
  extend $(\cI_3,\cI_4,\cI_5)$ to $T_2$ (in the sense that the
  restriction of $(\cI_3,\cI_4,\cI_5)$ to $C \cup N\cup T_2$ is a
  representation of the subgraph of $G$ induced by $C \cup N\cup T_2$) and
  completes the description of our 5-box representation
  $\cR=(\cI_1,\cI_2,\cI_3,\cI_4,\cI_5)$ of $G$.

\smallskip

We now prove that $\cR$ is a representation of $G$, i.e. $G$ is precisely the
intersection of the interval graphs $\cI_i$, for $1\le i \le 5$. By
the definition of $\cR$, the restriction of $\cR$ to
$N\cup R \cup T_1$ represents the subgraph of $G$ induced by
$N\cup R \cup T_1$ and the restriction of $\cR$ to $C \cup N \cup T_2$
represents the subgraph of $G$ induced by $C \cup N \cup T_2$. So we
only need to consider pairs of vertices $u\in C\cup T_2$ and
$v \in R\cup T_1$ (by definition, two such vertices $u,v$ are
non-adjacent in $G$), and show that the 5-boxes of $u$ and $v$ in
$\cR$ are disjoint. To see this, consider only $\cI_5$ and observe
that all vertices of $T_2$ are mapped to intervals that are at
distance at most $\tfrac14$ of an interval of $C$, so all the
intervals of $C\cup T_2$ end before
$2+\tfrac14+\tfrac14=\tfrac52$. Similarly, all the intervals of
$R\cup T_1$ start after $3-\tfrac14= \tfrac{11}4$. It follows that the
intervals of $u$ and $v$ are disjoint in $\cI_5$ and consequently, the
boxes of $u$ and $v$ are disjoint in $\cR$, as desired.
\end{proof}

It was proved in~\cite{EJ13} that toroidal graphs have boxicity at
most 7, while there are toroidal graphs of boxicity 4. The proof was
mainly based on the 
the following result of Schrijver~\cite{S93}: 

\begin{thm}{\cite{S93}}\label{thm:sch}
Every graph embedded in
the torus with face-width $k$ contains $\lfloor 3k/4 \rfloor$
vertex-disjoint (and homotopic) non-contractible cycles.
\end{thm}

Using some ideas of the proof of Theorem~\ref{thm:locpla}, we now
improve the bound on the boxicity of toroidal graphs from 7 to
6. Theorem~\ref{thm:locpla} implies that graphs embedded on the torus
with edge-width at least 40 have boxicity at most 5. We also decrease
this bound on the edge-width from 40 to 4.

\begin{thm}\label{thm:tor}
If $G$ is a graph embedded on the torus, then $G$ has boxicity at
most 6. If, moreover, $G$ has edge-width at
least 4, then $G$ has boxicity at most 5.
\end{thm}

\begin{proof}
  Let $G$ be a graph embedded on the torus. As before, we can assume
  without loss of generality that $G$ is a triangulation, and thus the
  face-width and the edge-width of $G$ are equal. If $G$ has edge-width at least 6, then by
  Theorem~\ref{thm:sch}, $G$ has 4 vertex-disjoint (and homotopic) non-contractible
  cycles. Let $C$ be one of them. We can assume that $C$ is chordless
  (see~\cite{EJ13}). Note that because of the 4 vertex-disjoint
  cycles, the subgraph of $G$ induced by $C$ and its neighborhood can
  be embedded in a cylinder, and is therefore planar. The exact same proof as that of
  Theorem~\ref{thm:locpla} (with $g=1$) then shows that $G$ has
  boxicity at most 5. 

\smallskip

Assume now that $G$ has edge-width at most 5 and at least 4, and let $C$ be a
shortest non-contractible cycle (in particular, $C$ is chordless).
Observe that $G-C$ is planar; in the remainder of the proof, we fix a
planar embedding of $G-C$. Since $G$ has no non-contractible triangles, all
the triangles of $G$ (and $G-C$) are contractible. Let $H$ be the
frame of (the fixed embedding of)
$G-C$. Let $R$ be the set of vertices of $H$, and let $T$ be the set of
vertices of $G$ lying in a non-facial triangle of (the fixed planar
embedding of) $G-C$. If one of the two faces of $G-C$ with vertex-set
$L(C)$ or $R(C)$ is contained in the interior of some triangle $uvw$
of $G-C$, then $uvw$ is homotopic to $C$ in $G$, and thus
non-contractible (which contradicts the fact that $G$ has edge-width
at least 4). It follows that $L(C)$ and $R(C)$ are included in $R$ and
consequently, no
vertex of $T$ is adjacent to a vertex of $C$. Let $(\cI_1, \cI_2)$ be a strict 2-box
  representation of $H=G[R]$ as in the proof of
  Theorem~\ref{thm:locpla} (since $G$ has edge-width at least 4, a
  proof similar to that of Theorem~\ref{thm:locpla} shows that such a
  representation exists).

We first extend this 2-dimensional
  representation to $C$ by mapping all the vertices of
  $C$ to a large 2-box containing all the 2-boxes of
  $(\cI_1, \cI_2)$. Recall that $C$ induces a cycle of length 4 or 5. We now partition
  the vertices of $C$ into three independent sets $S_3,S_4,S_5$, each containing
  one or two vertices. For any vertex $v$ of $G$, let $N_v$ denote the
  neighborhood of $v$ in $C \cup R$. For each $3\le i \le 5$, we denote by $\cI_i$ the
  interval graph with vertex-set $C \cup R$ depicted in
  Figure~\ref{fig:interval} (left) if $S_i=\{x\}$ or
  Figure~\ref{fig:interval} (right) if $S_i=\{x,y\}$. Observe that in
  $(\cI_3,\cI_4,\cI_5)$, $C$ induces a cycle, $R$ induces a clique,
  and the adjacency between $C$ and $R$ is the same as in
  $G$. 

\begin{figure}[htbp]
\begin{center}
\includegraphics[scale=0.9]{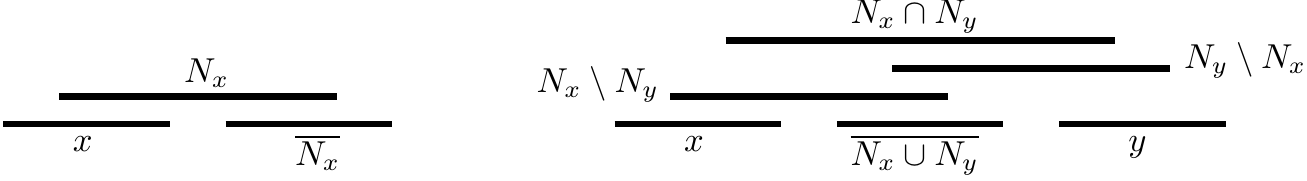}
\caption{The description of $\cI_i$ with $S_i=\{x\}$ (left), and the
  description of $\cI_i$ with $S_i=\{x,y\}$ (right). We use the
  following notation to avoid overloading the figure: $\overline{N_x}=C \cup
  R\setminus (N_x \cup \{x\})$ (left) and $\overline{N_x\cup N_y}=C \cup
  R\setminus (N_x \cup N_y \cup \{x,y\})$ (right).\label{fig:interval}}
\end{center}
\end{figure}

Since $C$ contains at most 5 vertices, one of the sets $S_i$ (say
  $S_3$) contains only one vertex, call it $x$. As before, we
  choose a small real $\epsilon_v>0$ for each vertex $v$ of $H$, so
  that all the chosen $\epsilon$ are distinct, and we change each
  interval $[s_v,t_v]$ of $v$ in $\cI_3$ to
  $[s_v+\epsilon_v,t_v+\epsilon_v]$. If each $\epsilon_v$ is small
  enough, this does not change the graph
  induced by $\cI_3$ (and thus the graph induced by $(\cI_3,\cI_4,\cI_5)$). Let $uvw$ be a triangle of
  $H$. Since $uvw$ is disjoint from $x$ (the unique vertex of
  $S_3$), it follows from the definition of $\cI_3$ (modified with the
  $\epsilon_v$) that the intervals corresponding to $u,v,w$ in $\cI_3$
  are strictly overlapping. By Lemma~\ref{lem:corner}, every
 triangle of $H$ has an empty inner corner in the strict 3-box representation
 $(\cI_1,\cI_2,\cI_3)$. By Corollary~\ref{cor:plaframe}, we can extend
  the strict 3-box representation $(\cI_1,\cI_2,\cI_3)$ of $H$ to $T$
  (so that in $\cI_3$, all the newly added intervals are disjoint from
  the interval of $x$). In $(\cI_4,\cI_5)$, it remains to map all the vertices of $T$ to a
  2-box of $(\cI_4,\cI_5)$ intersecting all the 2-boxes except that of
  the vertices of $C$. This can be done for instance by mapping in
  $\cI_i$ each vertex of $T$
  to the interval labelled $\overline{N_x}$ in
  Figure~\ref{fig:interval} (left), or the interval labelled
  $\overline{N_x\cup N_y}$ in
  Figure~\ref{fig:interval} (right). This ensures that in
  $(\cI_3,\cI_4,\cI_5)$, the vertices of $T$ are adjacent to all the other
  vertices of $G-C$, and non-adjacent to all the vertices of $C$, as
  desired. The representation $(\cI_1,\cI_2,\cI_3,\cI_4,\cI_5)$ of $G$
  then shows that $G$ has boxicity at most 5.

\medskip

Assume now that $G$ has edge-width at most 3. Then a set $S$ of at
most 3 vertices can be removed from $G$ so that $G-S$ is planar, and
thus has boxicity at most 3. Take a 3-box representation of $G-C$, and
extend it to $C$ by mapping all the vertices of $C$ to a large 3-box
containing all the 3-boxes of the representation. Now add 3 intervals
graphs, one for each element of $S$, defined as in
Figure~\ref{fig:interval} (left), where we now define $N_x$, with
$x\in S$, as the neighborhood of $x$ in $G$, and  $\overline{N_x}$ as
$V(G)\setminus (N_x \cup \{x\})$. The obtained 6-box representation
induces $G$, so $G$ has boxicity at most 6.
\end{proof}

\section{Linear extendability}\label{sec:linext}

Theorem~\ref{thm:locpla} easily implies that there exists a function
$f$, such that for any $g\ge 0$ and any graph $G$ embedded on $\OS_g$,
a set of at most $f(g)$ vertices can be removed from $G$ so that the
resulting graph has boxicity at most 5. However, the function $f$
derived from Theorem~\ref{thm:locpla} is exponential in $g$. In this
section, we show how to make $f$ linear in $g$. Note that the proof
works for graphs of Euler genus $g$ (while Theorem~\ref{thm:locpla}
is only concerned with graphs embeddable on $\OS_g$). The previously best known result of this type was that
$O(g)$ vertices can be removed in any graph of Euler genus $g$, so that the
resulting graph has boxicity at most 42~\cite{EJ13}.

\medskip

We will use a technique of Kawarabayashi and Thomassen~\cite{KT12},
who used it to prove several results of this type. Kawarabayashi and
Thomassen~\cite[Theorem~1]{KT12} proved that any graph $G$ embedded on
some surface of Euler genus $g$, with face-width more than $10t$ (for
some constant $t$) has a partition of its vertex-set into three parts
$A,P,X$, such that $X$ has size at most $10tg$, $P$ consists of the
disjoint union of paths that are local geodesics (in the sense that
each subpath with at most $t$ vertices of a path of $P$ is a shortest
path in $G$ and any two vertices at distance at least $t$ in some path
of $P$ are at distance at least $t$ in $G$) and are pairwise at
distance at least $t$ in $G$, and $A$ induces a planar graph having a
plane embedding $H$ such that the only vertices of $A$ having a
neighbor in $P$ lie on the outerface of $H$. 

\medskip

We will also use the
following technical lemma.

\begin{lem}\label{lem:tech}
Let $G$ be a graph whose vertex-set is partitioned into two sets $K$
and $P$, such that $K$ induces a complete graph, $P$ induces a path,
and for every vertex $u$ of $K$, the neighbors of $u$ in $P$ lie in a
subpath of $P$ of at most 3 vertices (equivalently, any two neighbors
of $u$ in $P$ are at distance at most two in $P$). Then for any real
number $t$, $G$ has a
3-box representation $(\cI_1,\cI_2,\cI_3)$ such that all the intervals
of $\cI_3$ corresponding to some vertex of $K$ end at $t$, while all the intervals
of $\cI_3$ corresponding to some vertex of $P$ end strictly before $t$.
\end{lem}

\begin{figure}[htbp]
\begin{center}
\includegraphics[scale=1.2]{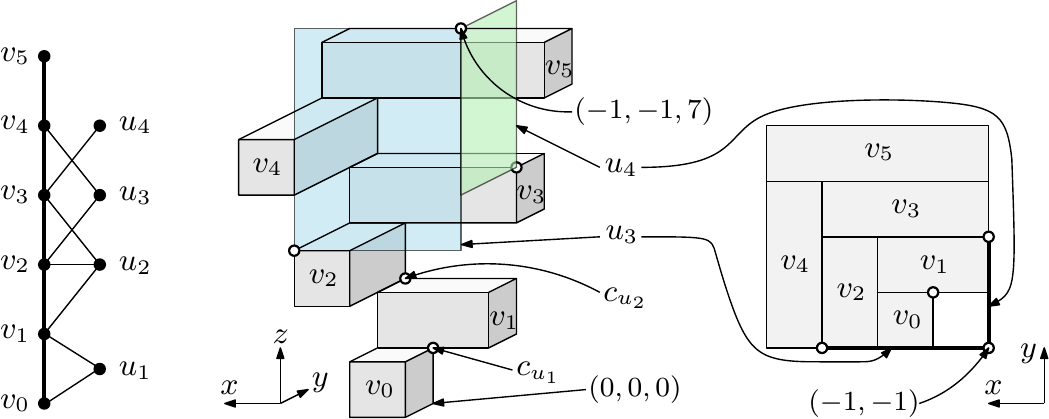}
\caption{The point $(-1,-1,7)$ and the bottom corners $c_{u_i}$ of the
  vertices $u_i$ are depicted with white dots. For the sake of
  readability, the 3-boxes of $u_1$ and $u_2$ are not displayed (only
  two of their corners are depicted). \label{fig:kp}}
\end{center}
\end{figure}

\begin{proof}
We first construct a 3-box representation of $G$, and then show how to
slightly modify it so that it satisfies the additional constraint on $\cI_3$.

  Let $P=v_0,v_1,\ldots,v_p$. For every $i\ge 0$, $v_{2i}$ is mapped
  to the 3-box $[i,i+1]\times [-1,i]\times [2i,2i+1]$ and $v_{2i+1}$
  is mapped to the 3-box $[-1,i+1]\times [i,i+1]\times [2i+1,2i+2]$
  (see Figure~\ref{fig:kp}, where both a 3-dimensional view and a
  2-dimensional view from above are depicted for the sake of clarity). Let $u$ be a vertex of $K$. Then $u$ is
  mapped to the 3-box with corners $(-1,-1,p+2)$ and $c_u$, where
  $c_u$ is defined as follows. If $u$ has no neighbor in $P$, then
  $c_u=(-1,-1,p+2)$. If $u$ has a single neighbor $v_j$ in $P$, then
  either $j=2i$ and we define $c_u=(i,-1,2i+1)$, or $j=2i+1$ and we
  define $c_u=(-1,i,2i+2)$ (see for example the 3-box of $u_4$ in
  Figure~\ref{fig:kp}). If the neighbors of $u$ are two consecutive
  vertices of $P$, say $v_{2i}$ and $v_{2i+1}$ (the case
  $v_{2i+1}, v_{2i+2}$ can be handled analogously by switching the
  roles of the $x$- and $y$-axis), then we set $c_u=(i,i,2i+1)$ (see
  for example the bottom corner $c_{u_1}$ of $u_1$ in
  Figure~\ref{fig:kp}). If the neighbors of $u$ in $P$ are $v_{2i}$
  and $v_{2i+2}$, then $c_u=(i,-1,2i+1)$ (see for example the 3-box of
  $u_3$ in Figure~\ref{fig:kp}). The case where the neighbors of $u$
  in $P$ are $v_{2i+1}$ and $v_{2i+3}$ is handled analogously by
  switching the roles of the $x$- and $y$-axis. Finally, if the
  neighbors of $u$ in $P$ are $v_{2i},v_{2i+1}, v_{2i+2}$, for some
  $i$, then we set $c_u=(i+1,i,2i+1)$. Again, the case where the
  neighbors of $u$ in $P$ are $v_{2i+1}, v_{2i+2},v_{2i+3}$ is handled
  analogously by switching the roles of the $x$- and $y$-axis (see for
  example the bottom corner $c_{u_2}$ of $u_2$ in
  Figure~\ref{fig:kp}).

  All the 3-boxes of the vertices of $K$ contain the point
  $(-1,-1,p+2)$, so $K$ induces a complete graph in the representation
  defined above. Moreover, it readily follows from the definition of
  the 3-boxes of the vertices $v_i$ and the corners $c_u$ that for
  each vertex $u$ of $K$, the neighbors of $u$ in $P$ are precisely
  the same in the graph $G$ and in the 3-box representation defined
  above. Consequently, the constructed representation induces $G$,
  as desired.

\smallskip

Let $\cI_1,\cI_2,\cI_3$ be the three interval graphs corresponding
respectively to the $x$-, $y$-, and $z$-axis in the representation
above. It follows from the construction of $\cI_3$ that all the
vertices $v \in K$ are mapped in $\cI_3$ to an interval of the form
$[i_v,p+2]$, while all the vertices of $P$ are mapped in $\cI_3$ to intervals
ending strictly before $p+2$. It is then easy to translate the whole
representation along the $z$-axis so that it satisfies the additional property.
\end{proof}

Finally, we will need the following direct consequence of
Theorem~\ref{thm:tho45v2}.

\begin{cor}\label{cor:plafixed}
  Let $G$ be a planar graph, let $w$ be a fixed vertex of $G$, and let
  $t\in \mathbb{R}$. Then
  $G$ has a strict 3-box representation $(\cI_1,\cI_2,\cI_3)$ such
  that the interval $I_w$
  of $w$ in $\cI_3$ ends at $t$, while all the other intervals of
  $\cI_3$ are contained in $[t,+\infty)$.
\end{cor}

\begin{proof}
We can assume without loss of generality that $G$ is a triangulation (since
it is an induced subgraph of some triangulation), and that $w$ belongs
to the outerface of $G$. We first map the three vertices of the outerface to 3-boxes as in
Figure~\ref{fig:corner}, and then apply Theorem~\ref{thm:tho45v2} to
extend this representation to a strict 3-box representation of $G$,
such that the boxes of all the internal vertices are inside the inner
corner $C$. Note that some hyperplane separates the box
$B_w$ of $w$ from the boxes of all the other vertices of $G$, and the
representation can therefore be translated in $\mathbb{R}^3$ in order to satisfy the
desired property.
\end{proof}

We are now able to prove the main result of this section.

\begin{thm}\label{thm:linex1}
Let $G$ be a graph of Euler genus $g>0$. Then $G$ contains a set $X$ of
at most
$60 g-30$ vertices such that $G-X$ has boxicity at most 5.
\end{thm}

\begin{proof}
 We prove the theorem by induction on $g>0$. If $G$ has face-width at
  most $30$, then $G$ contains a set $X$ of at most $30$ vertices such
  that $G-X$ has Euler genus at most $g-1$, or $G-X$ is the disjoint
  union of two graphs of Euler genus $g_1>0$ and $g_2>0$ with $g_1+g_2=g$  (see Proposition
  4.2.1 and Lemma 4.2.4 in~\cite{MoTh}). In the first case, either
  $G-X$ is planar (in which case the result clearly holds,
  since $G-X$ has boxicity at most 3 and $60g-30\ge 30$), or by the
  induction, a set $X'$ of at most $30+60 (g-1)-30\le 60g-30$ can be removed from $G$ in
  order to obtain a graph with boxicity at most 5. In the second case, by the
  induction, a set $X'$ of at most $30+(60 g_1-30)+(60 g_2-30)\le 60g-30$ can be removed from $G$ in
  order to obtain a graph with boxicity at most 5. As a consequence,
  we can assume that $G$ has face-width at least $30$, and apply the
  result of Kawarabayashi and Thomassen mentioned above, with
  $t=3$. 

Let $A,P,X$ be the corresponding partition of the vertex-set
  of $G$ (and let $H$ be the planarly embedded subgraph of $G$ induced
  by $A$, such that only the outerface $O$ of $H$ has neighbors in
  $P$). Note that $X$ contains at most $30 g \le 60g-30$ vertices, and we will
  prove that $G-X=G[P\cup A]$, the subgraph of $G$ induced by $A$ and $P$,
  has boxicity at most 5.

Let $H^+$ be the planar graph obtained from $H$ by adding a new vertex
$v^+$ adjacent to all the vertices of $O$. By
Corollary~\ref{cor:plafixed} (with $w=v^+$), $H^+$ has a strict 3-box representation
$(\cI_1,\cI_2,\cI^+_3)$ such that for some real number $p^+$, all the
intervals of $\cI^+_3$ corresponding to some vertex of $O$ start at
$p^+$, while all the intervals of $\cI^+_3$ corresponding to some vertex
of $A\setminus O$ start (strictly) after $p^+$.

Let $v$ be a vertex of $H$. Since the paths of $P$ are local
geodesics, and any two paths are at distance at least 3 apart, $v$ has at most 3 neighbors in $P$
and these neighbors lie on a subpath of at most 3 vertices of a path
of $P$ (i.e. they are either consecutive or at distance two on some
path of $P$). Let $P_1, P_2,\ldots,P_k$ be the paths of $P$, and for
each $1\le i \le k$, consider the two endpoints of $P_i$ and decide
arbitrarily which one is the left endpoint and which one is the right
endpoint. Let $\tilde{H}$ be the graph obtained from $G[P\cup O]$ by
adding, for each $1\le i < k$, a vertex $v_i$ adjacent (only) to the
right endpoint of $P_i$ and the left endpoint of $P_{i+1}$, and by
adding an edge between any two (non-adjacent) vertices of $O$. By
Lemma~\ref{lem:tech}, $\tilde{H}$ has a 3-box representation
$(\cI_4,\cI_5,\cI^-_3)$ such that the intervals of $\cI^-_3$ either end at
$p^+$ (if they correspond to a vertex of $O$), or end strictly before
$p^+$ (if they correspond to a vertex of $P$).
The restriction of $(\cI_4,\cI_5,\cI^-_3)$
to $P\cup O$ induces a 3-box representation of the graph obtained from
$G[P\cup O]$ by adding an edge between any two 
(non-adjacent) vertices of $O$, with $\cI^-_3$ satisfying the same
additional property as above.

Let $\cI_3$ be the interval representation obtained from $\cI^+_3$ and
$\cI^-_3$ as follows. Every vertex of $A\setminus O$ is mapped to its
image in $\cI^+_3$, every vertex of $P$ is mapped to its image in
$\cI^-_3$, and every vertex of $O$ is mapped to the concatenation of
its images in $\cI^-_3$ and
$\cI^+_3$ (note that the former ends at $p^+$ and the latter starts at
$p^+$). Note that the adjacency between $O$ and $P$ is the same in
$\cI_3$ and $\cI^-_3$, and the adjacency between $O$ and $A\setminus
O$ is the same in
$\cI_3$ and $\cI^+_3$. Moreover, the intervals of $A\setminus O$ are
disjoint from the intervals of $P$ in $\cI_3$.

It remains to map every vertex of $P$ in $(\cI_1,\cI_2)$ to a large
2-box containing all the other 2-boxes of $(\cI_1,\cI_2)$, and to map
every vertex of $A\setminus O$ in $(\cI_4,\cI_5)$ to a large
2-box containing all the other 2-boxes of $(\cI_4,\cI_5)$. Let
$\cR=(\cI_1,\cI_2,\cI_3,\cI_4,\cI_5)$. Note that the restrictions to
$A$ of
$\cR$, $(\cI_1,\cI_2,\cI_3)$, and $(\cI_1,\cI_2,\cI^+_3)$, all 
induce the same graph (namely, $G[A]$). Similarly, the restrictions to
$O \cup P$ of $\cR$, $(\cI_3,\cI_4,\cI_5)$, and $(\cI^-_3,\cI_4,\cI_5)$, all induce the same graph (namely, $G[O \cup P]$). Since the intervals of $A\setminus O$ are
disjoint from the intervals of $P$ in $\cI_3$, there are no edges
between $P$ and $A\setminus O$ in $\cR$. It follows that $R$
represents $G[A \cup P]=G-X$, as desired.
\end{proof}

\section{Large girth graphs}\label{sec:girth}

The previous sections were devoted to graphs embedded in fixed
surfaces, without short non-contractible cycles. Here we consider
graphs without short cycles at all. Using the results of~\cite{ACS14}
relating the boxicity of a graph and its second largest eigenvalue (in
absolute value), together with the existence of Ramanujan graphs of
arbitrarily large degree and girth~\cite{LPS88}, it directly follows that there is
a constant $c>0$ such that for any integers $d$ and $g$, there is a
$k$-regular graph ($k\ge d$) of girth at least $g$ and boxicity at
least $ck/\log k$. As a consequence, there are (regular) graphs with arbitrarily large girth
and boxicity. Therefore, in order to bound the boxicity of graphs without short
cycles, it is necessary to restrict ourselves to specific classes of
graphs. 

\smallskip

Let $p\ge 1$ be an integer. A graph $G$ is said to be
\emph{$p$-path-degenerate} (see~\cite{NO15}) if any subgraph $H$ of $G$ contains a vertex of degree
at most 1, or a path with $p$ internal vertices, each having degree
two in $H$. 

\smallskip

A 3-box representation of a graph $G$ is called a \emph{3-segment
  representation} if (1) each vertex is mapped to a \emph{segment}, (the
cartesian product of two points and an interval of positive length), (2)
the interiors of any two segments are disjoint (in other words, the interior of a
segment is only intersected by endpoints of other segments), and (3) no
two segments lie on the same line. We will prove the following result:

\begin{thm}\label{thm:6path}
Any 5-path-degenerate graph $G$ has a 3-segment representation.
\end{thm}

\begin{figure}[htbp]
\begin{center}
\includegraphics[scale=1.4]{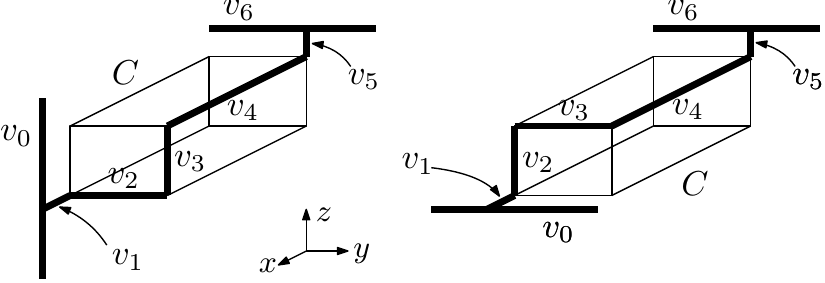}
\caption{The representation of a path with 5 internal vertices of degree
  2 between $v_0$ and $v_6$ when $S_{v_0}$ and $S_{v_6}$ are parallel to
  different axes (left)
  and when $S_{v_0}$ and $S_{v_6}$ are parallel to the same axis (right). \label{fig:5path}}
\end{center}
\end{figure}

\begin{proof}
We will prove the result by induction on
the number of vertices of $G$. Assume first that $G$ contains a vertex
$v$ of degree at most 1, and let $H=G-v$. Note that $H$ is
5-path-degenerate, so by the induction it has a 3-segment
representation $\cS=(S_v)_{v \in H}$. If $v$ has degree 0, then $G$ is the disjoint union of $H$
and $\{v\}$, and clearly has a 3-segment representation. Thus, we can
assume that $v$ has a unique neighbor $u$ in $G$. Since $\cS$
contains a finite
number of segments, $S_u$ contains a point $p$ such that some small
ball $B$ centered in $p$ only intersects $S_u$. We then represent $S_v$ as a
segment orthogonal to $S_u$ (there are two possible choices of dimension), with $p$ as one endpoint, and such that
$S_v$ lies inside $B$.

In the remainder of the proof, we assume that $G$ contains a path
$P=v_0v_1\ldots v_6$, such that for any $1\le i \le 5$, the only
neighbors of $v_i$ in $G$ are $v_{i-1}$ and $v_{i+1}$. Let $H$ be the
graph obtained from $G$ by removing all the vertices $v_i$ with $1\le
i \le 5$, and let $\cS=(S_v)_{v \in H}$ be a 3-segment representation
of $H$. We now extend $\cS$ to the vertices $v_1,v_2,\ldots,v_5$. For
$i=0,6$, fix
a point $p_i \in S_{v_i}$ such that some small ball $B_i$ centered in
$p_i$ only intersects $S_{v_i}$. Assume without loss
of generality that either $S_{v_0}$ is parallel to the $z$-axis and
$S_{v_6}$ is parallel to the $y$-axis (this includes the case where $v_0$
and $v_6$ are adjacent), or both $S_{v_0}$ and $S_{v_6}$ are
parallel to the $y$-axis (this includes the case where $v_0$
and $v_6$ are the same vertex). Assume that $B_0$ and $B_6$ have
radius at least $\epsilon$, for some $\epsilon>0$. We map $v_1$ to the segment $S_{v_1}$ that has $p_0$ as an
endpoint, is parallel to the $x$-axis, has length $\epsilon$, and goes
in the direction of $p_6$ (parallel to the $x$-axis). Similarly, we map $v_5$ to the segment $S_{v_5}$ that has $p_6$ as an
endpoint, is parallel to the $z$-axis, has length $\epsilon$, and goes
in the direction of $p_0$ (parallel to the $z$-axis). Let $p_1$ be
the endpoint of $S_{v_1}$ distinct from $p_0$, and let $p_5$ be
the endpoint of $S_{v_5}$ distinct from $p_6$. Let $C$ be the 3-box
with corners $p_1$ and $p_5$. We map $v_2$ to the edge $S_{v_2}$
of $C$ which contains $p_1$ and is orthogonal to $S_{v_0}$ and $S_{v_1}$, and similarly we
map $v_4$ to the edge $S_{v_4}$
of $C$ which contains $p_5$ and is orthogonal to $S_{v_5}$ and $S_{v_6}$. Finally, we map
$v_3$ to the edge $S_{v_3}$ of $C$ connecting the endpoint of
$S_{v_2}$ distinct from $p_1$ and the endpoint of
$S_{v_4}$ distinct from $p_5$ (see Figure~\ref{fig:5path}). Note that the
description of $S_{v_i}$, $1\le i \le 5$, above only depends of the
choice of $p_0$, $p_1$, and $\epsilon$. We now move each of $p_0$, $p_1$ and the value of $\epsilon$ along a tiny
interval. Then the locus described by each $S_{v_i}$, $1\le i \le 5$,
is a non-degenerate 3-box. Since $\cS$ is the union of a finite number
of segments, it follows that we can choose $p_0$, $p_1$, and
$\epsilon>0$ so that the $S_{v_i}$ ($1\le i \le 5$) are disjoint from
$\cS$. We can moreover choose $p_0$, $p_1$, and
$\epsilon>0$ so that $C$ is a non-degenerate 3-box (and so the $S_{v_i}$
are (non-degenerate) segments) and no segment $S_{v_i}$ lie on the
same line as a segment of $\cS$. Consequently, the obtained representation
is a 3-segment representation and each vertex $v_i$ with $1\le i \le 5$ is mapped to a segment
that only intersects the segments of $v_{i-1}$ and $v_{i+1}$, as desired.
\end{proof}

It was proved by Galluccio, Goddyn and Hell~\cite{GGH01} that for any
proper minor-closed class $\cF$ and for any $k$, there is an integer
$g=g(k)$ such that any graph of $\cF$ with girth at least $g$ is
$k$-path-degenerate. Since 3-segment representations are 3-box
representations, we have the following immediate consequence.

\begin{cor}\label{cor:minor}
For any proper minor-closed class $\cF$ there is an integer $g=g(\cF)$
such that any graph of $\cF$ of girth at least $g$ has boxicity at
most 3.
\end{cor}

Note that the result of  Galluccio, Goddyn and Hell was recently
extended by Ne\v{s}et\v{r}il and Ossona de Mendez~\cite{NO15} to
classes  of subexponential expansion, i.e. expansion bounded by $d
\mapsto \exp(d^{1-\epsilon})$, for some $\epsilon>0$ (see~\cite{NO15}
for definitions and further details). This shows that
Corollary~\ref{cor:minor} can be extended to this fairly broad setting
as well.

\smallskip

Interestingly, the bound on the boxicity in Corollary~\ref{cor:minor} is best possible already for the class of
$K_6$-minor free graphs. The following example was given by St\'ephan Thomass\'e. Take a copy of $K_5$, the complete
graph on 5 vertices, and replace each edge by an arbitrarily large
path. The resulting graph has arbitrarily large girth, no $K_6$-minor,
and any 2-box representation of it would give a planar embedding
(without crossings) of $K_5$, a contradiction.

\smallskip

For graphs of Euler genus $g$, Theorem 3.2 in~\cite{GGH01} (combined
with Theorem~\ref{thm:6path}) implies the following interesting counterpart of Theorem~\ref{thm:locpla}
(see the difference between the exponential bound there and the logarithmic
bound here).

\begin{cor}\label{cor:girth}
There is a constant $c$ such that any graph of Euler genus $g$ and
girth at least $c \log g$ has boxicity at
most 3.
\end{cor}

Observe that this is best possible up to the choice of the constant
$c$: for any integer $k$, there is a constant $c'=c'(k)$ and an infinite
family of graphs of (increasing) Euler genus $g$, girth at least $c' \log g$, and boxicity at
least $k$. This follows from the results of~\cite{ACS14} mentioned in
the introduction of this section, and the fact that the Ramanujan
graphs described in~\cite{LPS88} have girth logarithmic in their
number of vertices (and Euler genus linear in their number of
vertices, at least for $d$-regular graphs with $d\ge 7$).

\medskip

It is worth noting that Theorem~\ref{thm:6path} can also be applied to
classes that do not fit well in the framework of Ne\v{s}et\v{r}il and
Ossona de Mendez~\cite{NO15} (because their density is too high). Examples of such
classes include \emph{segment} or \emph{strings} graphs (intersection graphs of segments,
or strings in the plane), or \emph{circle} graphs (intersection graphs
of chords of a circle). For example, it can be proved
using Theorem~\ref{thm:6path} and the results of~\cite{EO} that every circle graph of girth at least 9 has
boxicity at most 3. This is in contrast with the existence of a circle
graph (indeed, a permutation graph) on $2n$ vertices with boxicity
$n$, for every $n\ge 1$ (see~\cite{Rob69}).

\section{Conclusion}\label{sec:ccl}

A natural problem is to find a counterpart of Theorem~\ref{thm:locpla}
for non-orientable surfaces. Non-orientable versions of
Theorem~\ref{thm:placc} exist~\cite{Yu97}, so the only problem when applying the
same arguments as in the
proof of Theorem~\ref{thm:locpla} to a graph embedded in a
non-orientable surface is that some of the cycles in the planarizing
collection might be one-sided, in which case the cycle $C$ together with
its neighborhood $N$ does not necessarily embed in a cylinder, but instead
on a M\"obius strip. As a consequence, these graphs are not necessarily
planar. However, using Lemma~\ref{lem:tech}, we can
prove that the graph obtained from the subgraph induced by $C\cup N$
by adding an edge between any two vertices of $N$, has boxicity at most
4 (just remove a vertex of $C$ and apply
Lemma~\ref{lem:tech}). Consequently, a proof along the lines of the proof
of Theorem~\ref{thm:locpla} easily shows that locally planar graphs
embedded on non-orientable surfaces have boxicity at most $4+2=6$.

\smallskip

It would be interesting to improve this bound, as well as that of
Theorem~\ref{thm:locpla}. It was conjectured in~\cite{EJ13} that locally planar graphs have
boxicity at most 3 (which would be best possible since there are
planar graphs of boxicity 3). We also conjecture the following
variant:

\begin{conj}
There is a constant $c>0$ such that in every graph embedded on a
surface of Euler genus $g$, at most $c g$ vertices can be removed so
that the resulting graph has boxicity at most 3.
\end{conj}

Note that the linear bound (in $g$) would be best possible, since
there are toroidal graphs with boxicity 4 (for example $K_8$ minus a perfect
matching, see~\cite{EJ13}), and the disjoint union of $\Omega(g)$ such
graphs can be embedded in a surface of Euler genus $g$.

\end{document}